\documentclass{article}

\usepackage{arxiv}
\usepackage[utf8]{inputenc} 
\usepackage[T1]{fontenc}    
\usepackage{hyperref}       
\usepackage{url}            
\usepackage{booktabs}       
\usepackage{amsmath, amssymb, amsfonts,amsthm} 
\usepackage{nicefrac}       
\usepackage{microtype}      
\usepackage{graphicx}       
\usepackage{subcaption}     
\usepackage{pgfplots}       
\graphicspath{ {./images/} }

\newtheorem{Definition}{Definition}
\newtheorem{Theorem}{Theorem}

\newtheorem{Remark}{Remark}
\newtheorem{Example}{Example}
\newtheorem{Corollary}{Corollary}

\title{Developable Ruled Surfaces Generated by the Curvature Axis of a Curve}

\author{
  Ferhat Ta\c{s} \\
  Department of Computer Science, Faculty of Science \\
  Istanbul University, 34143 Istanbul, T\"{u}rkiye \\
  \texttt{tasf@istanbul.edu.tr} \\
  \And
  Rushan Ziatdinov \\
  Department of Industrial Engineering, College of Engineering \\
  Keimyung University, 704-701 Daegu, Republic of Korea \\
  \texttt{ziatdinov@kmu.ac.kr} \\
}

\begin{document}

\maketitle

\begin{abstract}
Ruled surfaces play an important role in various types of design, architecture, manufacturing, art, and sculpture. They can be created in a variety of ways, which is a topic that has been the subject of much discussion in mathematics and engineering journals. In geometric modelling, ideas are successful if they are not too complex for engineers and practitioners to understand and not too difficult to implement, because these specialists put mathematical theories into practice by implementing them in CAD/CAM systems. Some of these popular systems such as AutoCAD, Solidworks, CATIA, Rhinoceros 3D, and others are based on simple polynomial or rational splines and many other beautiful mathematical theories that have not yet been implemented due to their complexity. Based on this philosophy, in the present work, we investigate a simple way to generate ruled surfaces whose generators are the curvature axes of curves. We show that this type of ruled surface is a developable surface and that there is at least one curve whose curvature axis is a line on the given developable surface. In addition, we discuss the classifications of developable surfaces corresponding to space curves with singularities, while these curves and surfaces are most often avoided in practical design. Our research also contributes to the understanding of the singularities of developable surfaces and, in their visualisation, proposes the use of environmental maps with a circular pattern that creates flower-like structures around the singularities.\end{abstract}

\keywords{Singularity, curvature axis, developable surface, ruled surface, differential geometry, Frenet frame, visualization of surfaces}

\footnote{
    "If people do not believe that mathematics is simple, it is only because they do not realize how complicated life is." \textit{— John von Neumann}
}

\section{Introduction}

The theory of curves and the theory of ruled surfaces are closely related topics, both of which are essential in space kinematics. A few years ago we came across a problem concerning the relationship between space curves and ruled surfaces for which we could find no study in the literature. Specifically, it is obvious and simple that the curvature axis of a space curve draws a ruled surface along the curve. However, this fact has not yet been studied in the literature. The only theory known so far that combines these two topics is E. Study's correspondence principle \cite{study}.

Ruled surfaces are a class of mathematical surfaces that can be generated by moving lines, also known as \textit{generators} \cite{2,Tas1,Tas2}. In CAD systems, such surfaces can be generated by moving the cross-section along the \textit{rail curve} \cite{Bo}. A developable surface is a special type of ruled surface that can be flattened into a plane without stretching or tearing. Due to their simplicity and ease of fabrication, they are used in a variety of fields, \cite{Wang,Pet,Li}\footnote{Interested readers can watch the \href{https://www.youtube.com/watch?v=-jZfFgfGj1M}{YouTube video} about applications of ruled surfaces in architecture.}.

Ruled surfaces are widely used in the modeling and design of complex shapes and structures while facilitating the analysis of algebraic surface properties, \cite{1,Fern,Lang,Rav}. In addition, their practical importance extends to various industries, including manufacturing and construction, where they provide a streamlined process for fabricating flat sheets of material. Ruled surfaces play a central role in mechanical engineering, contributing to the design of aircraft, automobiles, and applications ranging from agriculture to light manufacturing.


\begin{figure}
     \centering
     \begin{subfigure}[b]{0.45\textwidth}
         \centering
          \includegraphics[width=\linewidth]{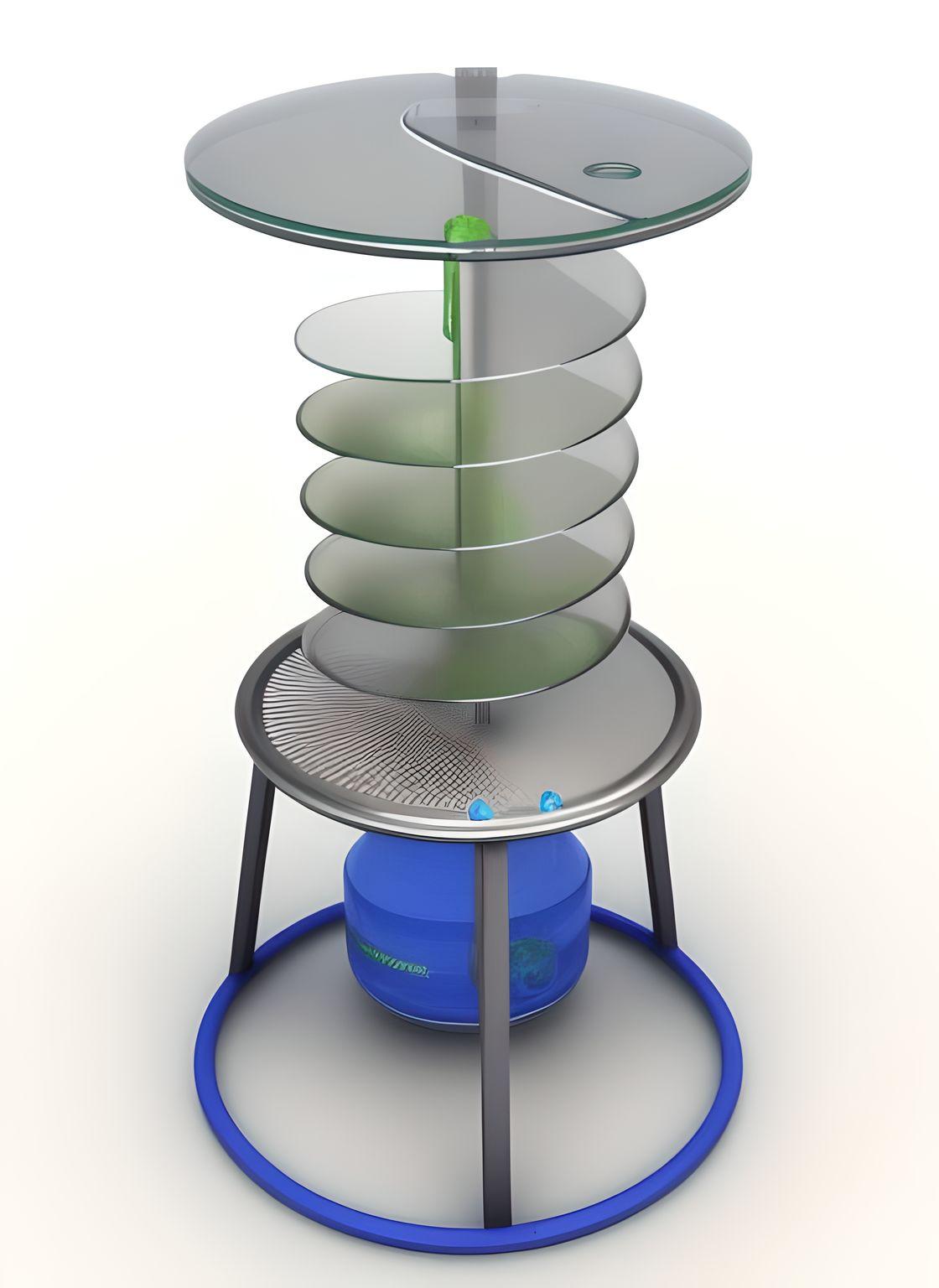}
        \caption{AI-enhanced image of a dew collector based on a ruled surface (helicoid) \cite{ziatdinov1}.}
     \end{subfigure}
     \hfill
     \begin{subfigure}[b]{0.45\textwidth}
         \centering
        \includegraphics[width=\linewidth]{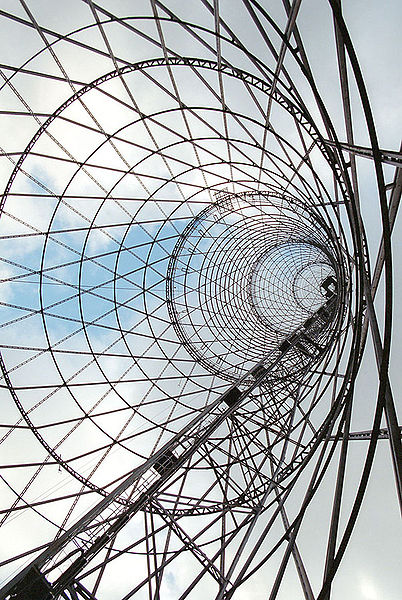}
        \caption{The grid shell of the Shukhov Tower in Moscow, whose sections are doubly ruled. The image has been adapted from Wikipedia.}
     \end{subfigure}
     \caption{\quad   Applications of ruled surfaces.}
        \label{fig:three graphs}
\end{figure}

\begin{figure}
        \centering
        \includegraphics[width=\linewidth]{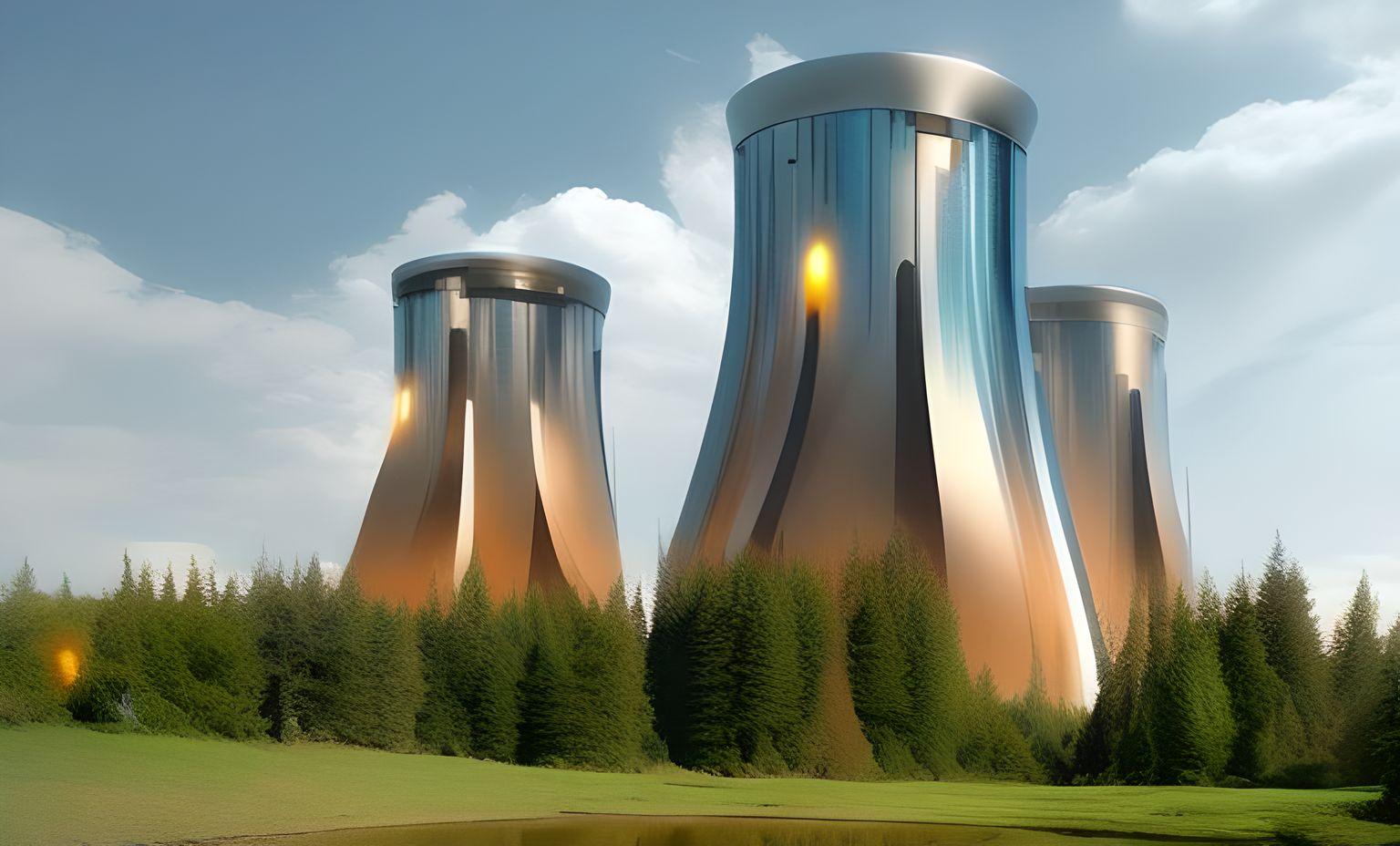}
        \caption{AI image generated from the Wikipedia photo of hyperbolic cooling towers (the surface can be doubly ruled) at Didcot Power Station, UK.}
\end{figure}

Discussing the importance and applications of developable surfaces, the paper \cite{Zha} reviews techniques and patents in geometric modelling and highlights developments in this field and concludes that further research is needed in several key areas of developable surfaces. Some special curves construct ruled surfaces with the help of the Frenet frame and there are also very good papers on this subject such as \cite{izumiya, lawrence, pottman1, tang, liu, glaeser, ishikawa1, chalfant, At}.

The main references on developable surfaces, \cite{2,izumiya,lawrence,pottman1}, have shown that developable surfaces can be created by moving lines across the surfaces. In addition to differential geometric properties of developable surfaces, such as Gaussian and mean curvature, singularities are also studied. The main papers on singularities are \cite{3,izumiya,ishikawa1}.
However, there is no information on developable surfaces that can be created from the curvature axis of a curve.

The present work aims to fill this gap by investigating the properties of the developable surfaces generated by the curvature axis and their properties in differential geometry. The idea is to exploit the relationship between the curvature axis and the developable surface to obtain a new type of ruled surface with interesting properties. We provide a result similar to this \textit{correspondence principle} \cite{study}, but between space curves and developable surfaces. Our work explores the relationship between the curvature axis of a curve and the developable ruled surface it generates, filling a gap in the existing literature and contributing to the understanding of the relationship between curves and developable surfaces.

The paper is structured as follows: In Section \ref{preliminary}, we provide some preliminary definitions and concepts related to curves, curvature axis, ruled surfaces, and developable surfaces. In Section \ref{dev_ruled_surface}, we show that moving the curvature axis along a curve yields a developable surface. In Section \ref{curve_alpha}, we discuss finding \(\alpha\) from some types of developable surfaces. Finally, in Section \ref{singularities}, we study the singularities of the developable surface and conclude our work in Section \ref{conclusion}.

\section{Preliminary Concepts and Definitions}\label{preliminary}

This section provides a brief overview of some basic concepts and definitions related to space curves and their kinematic frames. In particular, we define what a curve is and how to calculate its curvature and torsion. We also define what a ruled surface is and give an example of a developable surface. Finally, we introduce the concept of a curvature axis, its properties, and its relationship to the Frenet framework.

\begin{Theorem}
    (\textit{The Frenet-Serret Formulas}) Let \(\alpha\) be a regular and smooth curve whose first and second derivatives (i.e., velocity and acceleration vectors) are not required to be proportional. Then there is a rigid motion on the curve $\alpha$ defined by its Frenet frame, \cite{2,3,Sch,struik,Som,UmYa,A.Gray}:
$$
\begin{aligned}
\mathbf{t}^{\prime} &=\quad \kappa \mathbf{n} \\
\mathbf{n}^{\prime} &=-\kappa \mathbf{t} \quad+\tau \mathbf{b} \\
\mathbf{b}^{\prime} &=-\tau \mathbf{n},
\end{aligned}
$$
\end{Theorem}
\noindent where $\kappa \neq 0$ and $\tau$ are the curvature and torsion functions of the curve $\alpha$, respectively, and \((\prime:=\frac{d}{ds})\) refers to the derivative of the curve with respect to its arclength parameterisation.
The curvature at a given point on a differentiable curve is essentially the curvature of the osculator circle, the circle that provides the best approximation of the curve near that point. This close relationship of the circle to curvature is due to the fact that it has the same curvature at the same point.

\begin{Definition}
An osculator circle is a planar geometric object. Therefore, it is determined by the elements {t,n} that stretch the tangent plane of the curve, which are the elements that determine the planar behavior of the curve at the point where it touches the curve, and the curvature of the curve, which is inversely proportional to the distance from that point to the center of the osculator circle. The center of the osculator circle at a given point $s_0$ on the curve $\alpha$ is as follows, \textcolor{blue}{\cite{2,3,Sch,Som}}:

\begin{equation}
    \mathrm{C}=\alpha\left(s_{0}\right)+\frac{1}{\kappa\left(s_{0}\right)} \mathbf{n}\left(s_{0}\right).
\end{equation}
\end{Definition}

\begin{Definition}
The tangents of a curve $\alpha$ form a surface along the curve. This surface is called the tangent surface. The lines on this tangent surface that intersect perpendicular to the tangents of $\alpha$ are called the involutes, say $\beta$, of the curve $\alpha$ and can be given as follows:
\begin{equation}
    \beta(s)=\alpha(s)+ (c-s)\mathbf{t}(s).
\end{equation}
where $c$ is constant, \textcolor{blue}{\cite{Sch,struik,Som}}.
\end{Definition}
Conversely, $\alpha$ is the evolute of $\beta$. But there is a difference; the fact that $\beta$ is perpendicular to the tangent of $\alpha$ means that its tangent lies in the normal plane of $\alpha$. So the equation of evolute is given by 
\begin{equation}
    \alpha(s)=\beta(s)+\frac{1}{\kappa(s)}\mathbf{n}(s)+\frac{1}{\kappa(s)}\cot{\left(\int{\tau(s) ds+c}\right)} \mathbf{b}(s).
\end{equation}
These curves are now \textit{involute-evolute} mates.

\begin{Definition}
At a point $s_0$ of the curve $\alpha$, there are spheres whose order of contact with the curve is greater than 3, and these spheres are called osculator spheres, \cite{Sch,Som}:

\[
\mathbf{m_0}=\alpha\left(s_{0}\right)+\frac{1}{\kappa\left(s_{0}\right)} \mathbf{n}\left(s_{0}\right)-\frac{\kappa'(s_0)}{\kappa(s_0)^2\tau(s_0)} \mathbf{b}\left(s_{0}\right).
\]
\end{Definition}

The locus of the centers of these spheres generates a line called the \textit{curvature axis}, which passes through the center of curvature of the curve.

\begin{Definition}
The curvature axis of a space curve \(\alpha\) is represented by the following equation:
\begin{equation}
\mathbf{d} =\alpha\left(s_{0}\right)+\frac{1}{\kappa\left(s_{0}\right)} \mathbf{n}\left(s_{0}\right)+\lambda \mathbf{b}\left(s_{0}\right).
\end{equation}
\end{Definition}

In this formula:

\begin{itemize}
\item \(\mathbf{d}\) represents a line at the point $\alpha\left(s_{0}\right)+\frac{1}{\kappa\left(s_{0}\right)} \mathbf{n}\left(s_{0}\right)$.
\item \(\alpha(s_0)\) is a position vector associated with a point on a curve at a particular parameter value \(s_0\).
\item \(\frac{1}{\kappa(s_0)} \mathbf{n}(s_0)\) is the normal component, where \(\kappa(s_0)\) represents the curvature of the curve at the point \(s_0\), and \(\mathbf{n}(s_0)\) is the unit normal vector at that point. This component ensures that \(\mathbf{d}\) is displaced from the curve in the direction of its normal.
\item \(\lambda \mathbf{b}(s_0)\) is the binormal component, where \(\mathbf{b}(s_0)\) is the unit binormal vector at the point \(s_0\), and \(\lambda\) is a scaling factor that determines how far \(\mathbf{d}\) is from the curve along the binormal direction \cite{Sch,struik, Som, HHH}.
\end{itemize}

This formula is often used to describe the position of a point near a curve, considering its position on the curve (\(\alpha(s_0)\)), the effect of curvature (\(\frac{1}{\kappa(s_0)} \mathbf{n}(s_0)\)), and the binormal direction (\(\lambda \mathbf{b}(s_0)\)). It's a part of the mathematics used to study the geometry of curves in 3D space. Generally speaking, there is a correspondence between a point lying on the curve and a line (see Fig.\ref{fig:cur axs}).

\begin{figure}[]
\centering
\includegraphics[width=0.7\textwidth]{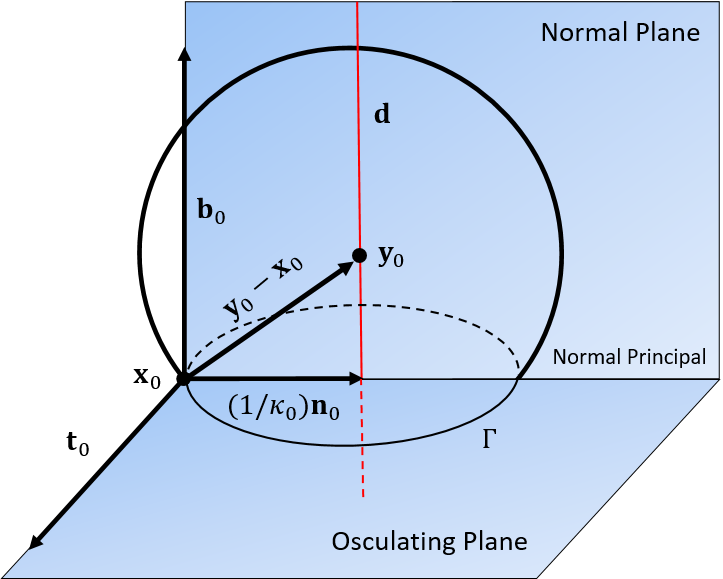}
\caption{The curvature axis (in red) of a regular curve $\alpha$ and its Frenet frame.}
 \label{fig:cur axs}
\end{figure}

\begin{Definition} A ruled surface is a surface $\mathbf{R}$ that admits a parameterization of the form
$$
\mathbf{R}(u, v)=\mathbf{x}(u)+v \mathbf{y}(u), \qquad u,v \in \mathbb{R},
$$
where $\mathbf{x}(u)$ and $\mathbf{y}(u)$ are curves called the base curve and directrix of $\mathbf{R}$ and \(v \in \mathbb{R}\), respectively, with $||\mathbf{y}(u)||=1$, \textcolor{blue}{\cite{2}}.
\end{Definition}

If $u$ is fixed, i.e. $u=u_0$ then $\mathbf{R}(u_0, v)$ generates a line $\mathbf{L}_{u_0}$ that passes through $\mathbf{x}(u_0)$ and has unit direction vector $\mathbf{y}(u_0)$. These lines are called the rulings of $\mathbf{R}$. If $\mathbf{y}(u)$ is constant, then $\mathbf{R}$ is called a cylinder. If $\mathbf{y}^{\prime}(u)$ never vanishes, then $\mathbf{R}$ is called a non-cylindrical surface.

A parameterisation for non-cylindrical ruled surfaces can also be given by \textit{striction curves}, \cite{pottman1}. This curve does not depend on the base curve on the surface and does not have to be regular. If we want to re-parameterise the developable surface with this curve, then we have

\[
\mathbf{R}(s,\lambda)=\mathbf{x}(s)+\lambda \mathbf{y}(s).
\]

Note that \(\mathbf{y}(s)\) is parallel to $\mathbf{x}'(s)$. 

\begin{Remark} The characterization of a developable surface can be determined by the following equation, \cite{pottman1}:

\begin{equation}\label{dc}
\det(\mathbf{x}^{\prime}(s), \mathbf{y}(s), \mathbf{y}^{\prime}(s))=0.
\end{equation}
\end{Remark}
The equation \(\det(\mathbf{x}^{\prime}(s), \mathbf{y}(s), \mathbf{y}^{\prime}(s))=0\) is a condition that must be satisfied for a curve parameterized by $s$ to be a trajectory of a certain type of differential equation. This type of differential equation is known as a Hamiltonian system \cite{Gold, Fow,Stg}.\\
Helicoid, which has many applications in engineering \cite{ziatdinov1}, is a very good example of this type of developable surface. It is also known as the tangent-developable surface of a circular helix curve.

\section{Construction of the Developable Ruled Surface}\label{dev_ruled_surface}

Here we investigate the geometric properties of the developable surface that is generated by the curvature axis of a curve and how these properties are related to the curvature and torsion of the curve.
The developable surface can be described by several parameters, such as its generators, directrices, and ruling lines, which can be calculated using the curvature and torsion of the curve.
Curvature theory encompasses several closely related concepts in geometry. Essentially, curvature describes the extent to which a curve differs from being perfectly flat or a surface from being perfectly flat. Flat surfaces are generally known as developable surfaces.

The relationship between the parameters of the developable surface and the curvature and torsion of the curve is important because it helps us understand the behavior of the surface and how it changes with variations in the shape of the curve. For example, changes in the curvature of the curve can affect the shape of the developable surface and its parameters, which in turn can affect its suitability for various applications such as architecture, product design, mechanical materials, manufacturing, etc.

We give the derivation of the equations for the curvature axis of a curve. If we consider this axis (line) through the curve, then we get a ruled surface.

\begin{Theorem}: Let \(\alpha=\alpha\left(s\right)\) be a smooth curve and \(s\) be arc-length parameterization and \(\{\mathbf{t},\mathbf{n},\mathbf{b},\kappa,\tau\}\) be its Frenet - Serret apparatus, then the trajectory

\begin{equation}\label{ca}
\mathbf{R}(s,\lambda)=\alpha\left(s\right)+\frac{1}{\kappa\left(s\right)} \mathbf{n}\left(s\right)+\lambda \mathbf{b}\left(s\right)
\end{equation}
generates a developable surface.
\end{Theorem}

\begin{proof} 
In order for \(R(s,\lambda)\) to be a developable surface, it must satisfy equation \eqref{dc}:
In terms of fitting into the equation,

\[
\mathbf{x}(s)=\alpha(s)+\frac{1}{\kappa(s)}\mathbf{n}(s),\qquad \mathbf{y}(s)=\mathbf{b}(s).
\]

\noindent Using the Frenet-Serret kinematic formulation given for the curve $\alpha$, we get

\[
\mathbf{x}'(s)=\mathbf{t}(s)+\frac{\kappa'(s)}{\kappa(s)^2}\mathbf{n}+\frac{1}{\kappa(s)}\left(-\kappa(s)\mathbf{t}(s)+\tau(s)\mathbf{b}(s)\right),\qquad \mathbf{y}'(s)=-\tau\mathbf{n}(s).
\]

\noindent So the developability condition
\[
\det(\mathbf{x}'(s),\mathbf{y}(s),\mathbf{y}'(s))=0
\]
is fulfilled.
\end{proof}

\section{Finding the curve \(\alpha\)}\label{curve_alpha}
What we have done so far is to find the developable surfaces corresponding to the given space curves. Now let us try to find the curve \(\alpha\) from a given ruled surface. For this, we need to look at the classes of developable surfaces.\\
In general, except planes, there are three types of developable ruled surfaces, \cite{A.Gray}\footnote{Readers interested in the classification of non-developable ruled surfaces are referred to the work \cite{hliu}.}. For each classification, we assume the equation of the ruled surface as follows:
\begin{equation}\label{frst}
    \mathbf{R}(s,\lambda)=\mathbf{x}(s)+\lambda\mathbf{e}(s).
\end{equation}
All comparisons below are made according to equation \eqref{frst}.

\subsection{Generalized Cylinder-developable surfaces}

\begin{equation}\label{cyd}
    \mathbf{R}(s,\lambda)=\mathbf{x}(s)+\lambda\mathbf{e}
\end{equation}
where the vector $\mathbf{e}$ is constant with the identities \(\mathbf{e} \neq 0, ||\mathbf{e}||=1\) and \(\mathbf{x}(s): I \subset \mathbb{R} \rightarrow \mathbb{R}^3\) respectively.
\subsection{Conical-developable surfaces}
\begin{equation}\label{cod}
    \mathbf{R}(s,\lambda)=\mathbf{x}+\lambda\mathbf{e}(s)
\end{equation}
where \(\mathbf{e}(s): I \subset \mathbb{R} \rightarrow \mathbb{R}^3\) is a smooth unit vector field and $\mathbf{x}$ is a point in \(\mathbb{R}^3\).
\subsection{Tangent-developable surface}
These surfaces are generated by moving a straight line that is parallel to the tangent along the curve.

\begin{equation}\label{tds}
\mathbf{R}(s,\lambda) = \mathbf{x}(s)+\lambda\mathbf{x}'(s)    
\end{equation}
of the curve \(\mathbf{x}(s)\) (\(\mathbf{x}(s) \in C^2\)) i.e. the curve of regression.

\subsection{Comparisons}
To determine whether the surfaces generated by the curvature axis of a curve fit with cylindrical, conical, or tangent surfaces, we can compare equation \eqref{ca} with the equations for these types of surfaces. This allows us to examine the geometric properties of the surfaces and determine their compatibility with these types of surfaces.

\begin{itemize}
\item If the surface is cylindrical, then from \eqref{ca} and \eqref{cyd} we get
\end{itemize}

\[
\mathbf{x}(s)=\alpha\left(s\right)+\frac{1}{\kappa\left(s\right)} \mathbf{n}\left(s\right), \quad \mathbf{b}(s)=\mathbf{e}.
\]

The binormal of the curve \(\alpha(s)\) is a constant vector, which means that \(\alpha(s)\) is a planar curve, i.e., \(\tau=0\).
\begin{itemize}
\item If the surface is conic-developable, then we get from the equations \eqref{ca} and \eqref{cod} following differential equation,
\end{itemize}

\[
\alpha''(s)+\alpha(s)=0.
\]

Solving this differential equation gives us the following coordinate functions:

\[
\alpha_i(s)=c_{i1}\sin{s}+c_{i2}\cos{s}
\]

\noindent where \(c_{i1} \text{ and } c_{i2}\) are constants.
\begin{itemize}
\item If the surface is of the tangential type, then from equations \eqref{ca} and \eqref{tds}, we get the following relations:
\end{itemize}

\begin{equation}\label{1}
\mathbf{x}(s) = \alpha(s)+\frac{1}{\kappa(s)} \mathbf{n}(s)
\end{equation}

\noindent and
\begin{equation}\label{2}
\mathbf{x}'(s) = \mathbf{b}(s).
\end{equation}

So, by differentiating equation \eqref{1} and substituting it into equation \eqref{2}, we get

\[
\mathbf{t}(s)-\frac{\kappa'(s)}{\kappa(s)^2} \mathbf{n}(s)+\frac{1}{\kappa(s)}[-\kappa(s)\mathbf{t}(s)+\tau(s)\mathbf{b}(s)] =\mathbf{b}(s)
\]
then

\[
-\frac{\kappa'(s)}{\kappa(s)^2} \mathbf{n}(s)+\frac{\tau(s)}{\kappa(s)}\mathbf{b}(s)=\mathbf{b}(s)
\]
resulting

\[
\kappa(s)=const, \quad \tau=\kappa.
\]
This means that the curve $\alpha$ is a slant helix. If we rewrite equation \eqref{1}, we get the following second-order linear ordinary differential equation (ODE) with variable coefficients as follows:

\[
\frac{1}{\kappa^2}\alpha''(s)+\alpha(s)=\mathbf{x}(s).
\]
Assuming $\kappa \neq 0$, differential equations for all coordinate functions $\alpha_i(s)$ of the curve $\alpha$ we get

\[
\alpha''(s)+\kappa^2 \alpha(s)=\kappa^2 \mathbf{x}(s).
\]
Solving these differential equations, we obtain the coordinate functions $\alpha_i(s)$ of the curve \(\alpha(s)\) as follows:

\begin{equation}
\alpha_{{i}} \left( s \right) =\sin \left( \kappa s \right) {C_{i1}}+\cos \left( \kappa s \right) {C_{i2}}+\kappa \left(\sin \left( \kappa s \right) \int \!\cos \left( \kappa s \right) x_{{i}} \left( s \right) {ds} -\cos \left( 
\kappa s \right) \int \!\sin \left( \kappa s \right) x_{{i}} \left( s \right) {ds} \right)
\end{equation}

\noindent where $C_{i1}, C_{i2}$ are constants.

\begin{Corollary} Thus, if we have a developable surface, there exists at least one curve \(\alpha\) for which every line on the developable surface can serve as a curvature axis at every corresponding regular point. There are, of course, many more if constant numbers \(C_i\) change.
\end{Corollary}

\begin{Example} We give a tangent-developable surface example. Let \(\mathbf{x}(s)\) be a standard cylindrical helix:
\[
\mathbf{x}(s) = (\cos(s), \sin(s), s)
\]
where its curvature is \(\kappa = \frac{1}{2}\). Constants are as follows: \(c_{11}=1,c_{12}=1,c_{21}=1,c_{22}=0,c_{31}=1,c_{32}=0\). Then from the equations of (6), we have

\begin{align*}
\alpha_1(s)&=-\frac{2 \left(\cos^{2}\left(\frac{s}{2}\right)\right)}{3}+\cos \! \left(\frac{s}{2}\right)+\sin \! \left(\frac{s}{2}\right)+\frac{1}{3},\\ 
\alpha_{2}( s)&=-\frac{2 \sin \! \left(\frac{s}{2}\right) \cos \! \left(\frac{s}{2}\right)}{3}+\sin \! \left(\frac{s}{2}\right),\\ 
\alpha_{3}(s)&=s +\sin \! \left(\frac{s}{2}\right).
\end{align*}

Thus, as can be seen in Fig.\ref{fig:tds}, the \(\alpha(s)\) curve corresponding to the \(\mathbf{x}(s)\) curve is found.
\end{Example}

\begin{figure}
\centering
\includegraphics[width=0.7\textwidth]{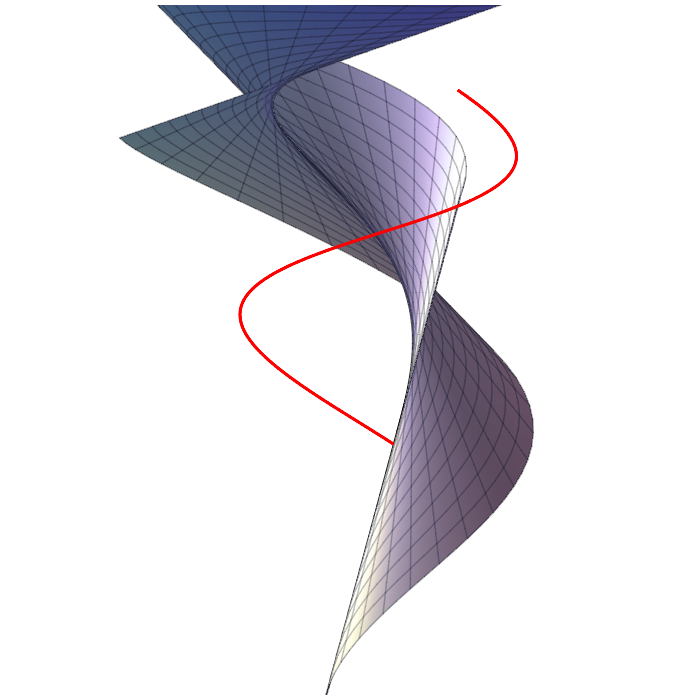}
\caption{The tangent developable surface corresponding to $\mathbf{x}(s)$(red one).}
\label{fig:tds}
\end{figure}

\section{Singularities}\label{singularities}

A singular point on a surface is a point at which the usual rules of calculus, such as continuity or differentiability, break down. Singularities are crucial to understanding the behavior of surfaces, especially in the context of deformations and transformations, and are given special attention in CAGD. Fig. \ref{singularity} shows an example of a surface with a cusp singularity created in Rhinoceros 3D software.

\begin{figure}
  \centering
  \includegraphics[width=0.9\textwidth]{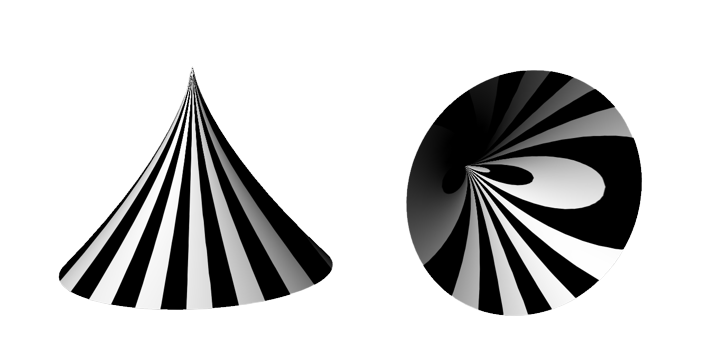}
  \caption{A surface of revolution with a cusp singularity created by rotating a cuspidal NURBS curve around an axis. Zebra stripes on the outer and inner parts of the surface behave in a specific way around the singularity.}
  \label{singularity}
\end{figure}

Another good way to illustrate singularities on surfaces is to use an environment map with uniformly distributed white circles on a black background. In this case, the circles will form flower-like patterns around the singularity (Fig. \ref{singularity_circ}). An environment map is a surface analysis tool used in CAD systems to visually evaluate surface smoothness, curvature and other important properties using an image reflected in the surface. In principle, any image can be used as an environment map, and readers interested in environment maps in Rhinoceros 3D software are referred to \cite{mcneel2023}. Also, readers should not confuse environment maps with environments that have the *.renv file extension.

\begin{figure}
  \begin{minipage}{0.3\textwidth}
    \centering
    \includegraphics[height=4.5cm]{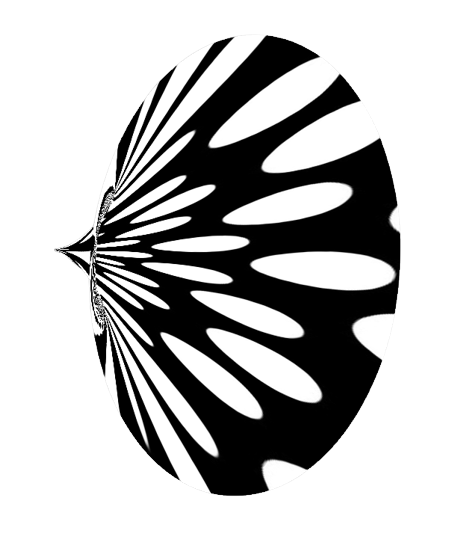} 
  \end{minipage}%
  \begin{minipage}{0.7\textwidth}
    \centering
    \includegraphics[height=4.5cm]{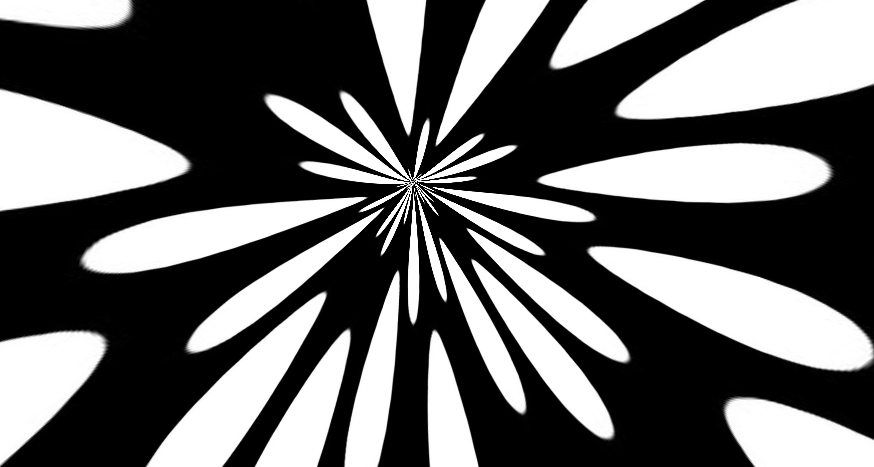} 
  \end{minipage}
  \caption{Environmental map with a circular pattern applied to the surface with a singularity. White circles tend to be elongated and increase in number around a single point, forming a flower-like pattern.}
  \label{singularity_circ}
\end{figure}

Combining these concepts, singularities of a developable surface refer to specific points on a developable surface where the surface can't be flattened smoothly. These points are particularly interesting in applications involving surface unfolding, such as paper folding, sheet metal fabrication, or computer-aided design.

There are several types of developable surface singularities. For example, cusps occur when a surface folds back on itself, creating a sharp point. Imagine folding a piece of paper; the folded corner forms a bump. The other is self-intersection. These occur when a surface intersects itself. In the context of developable surfaces, this means that the surface folds in such a way that parts of it intersect.

Developable surfaces with singularities are investigated in \cite{ishikawa1,3}.
There are three types of singularities: cuspidal edge, swallowtail, and cuspidal cross-cap, see Fig. \ref{singularity_circ}.

\begin{figure}
  \begin{minipage}{0.3\textwidth}
    \centering
    \includegraphics[height=5cm]{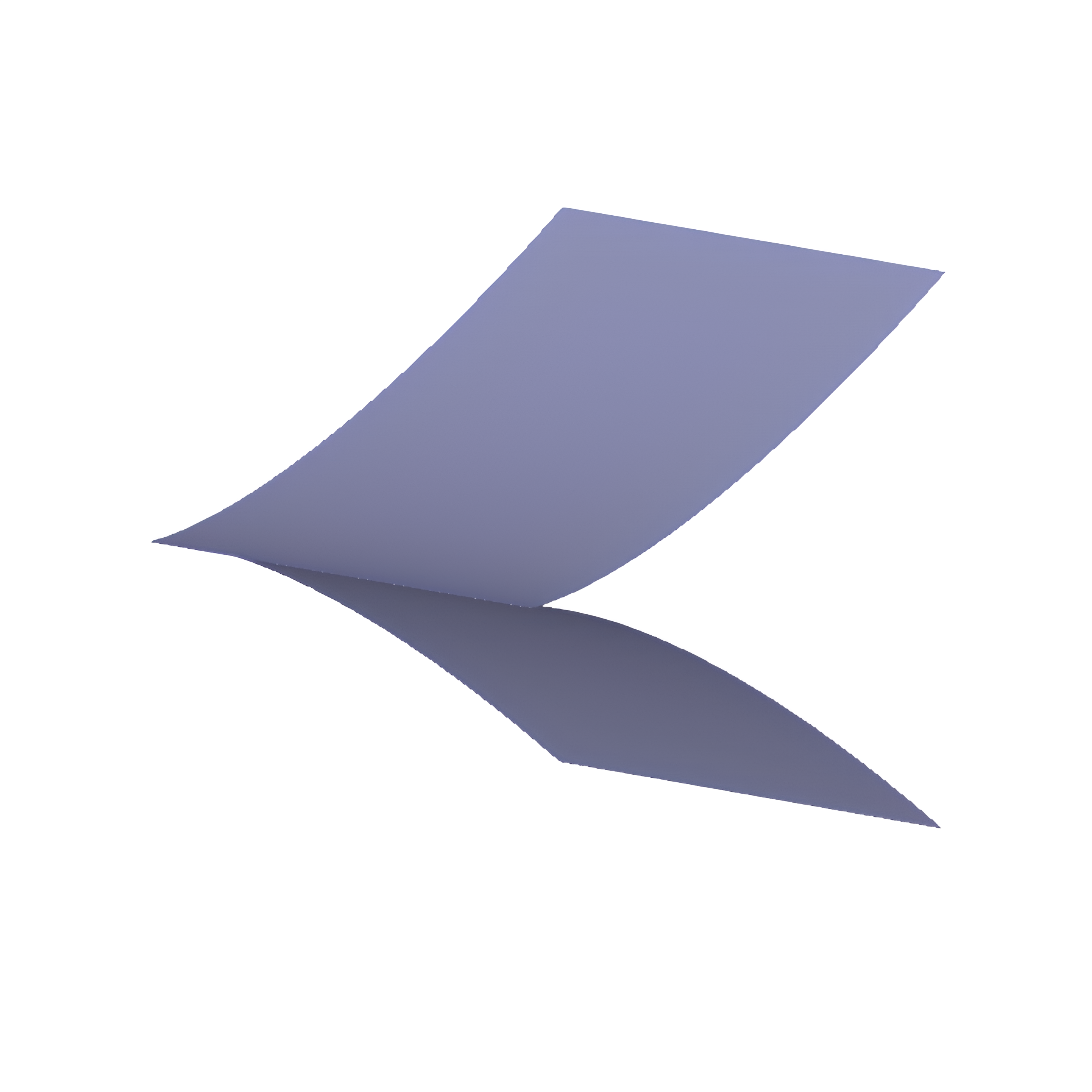} 
  \end{minipage}%
  \begin{minipage}{0.3\textwidth}
    \centering
    \includegraphics[height=5cm]{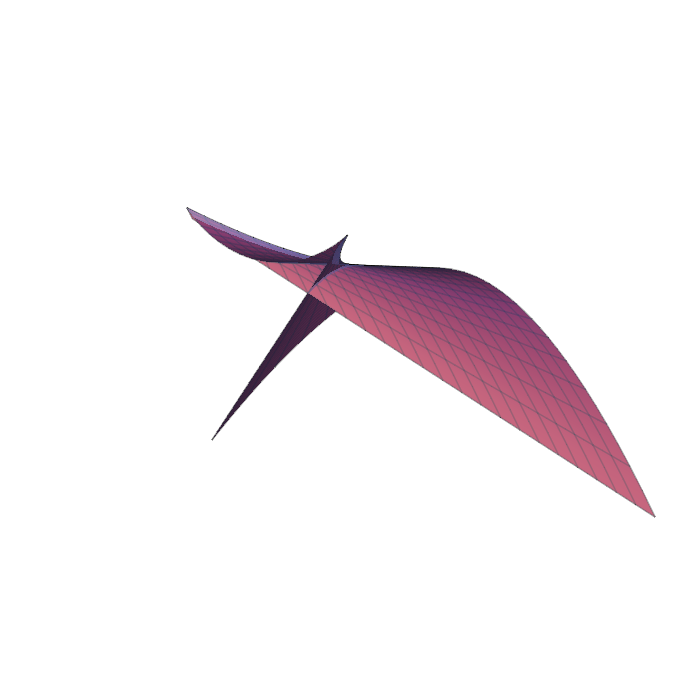} 
  \end{minipage}%
  \begin{minipage}{0.3\textwidth}
    \centering
    \includegraphics[height=5cm]{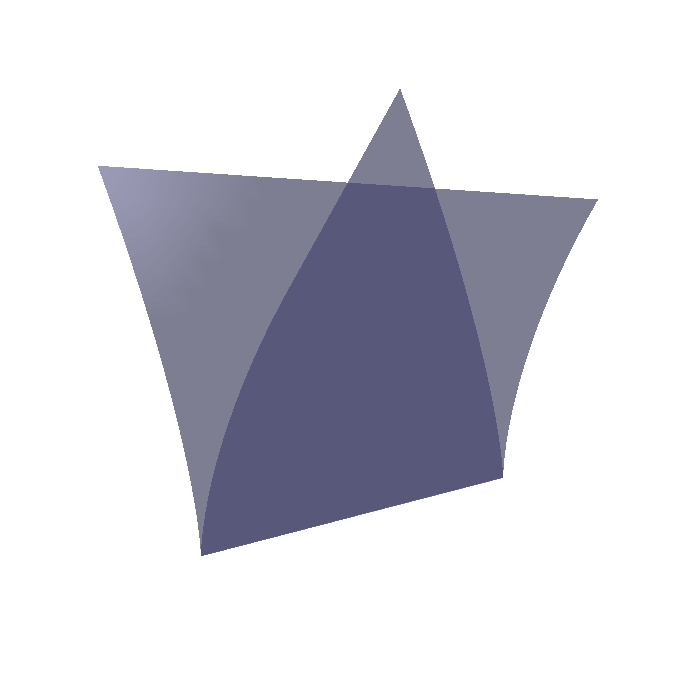} 
  \end{minipage}
  \caption{Singular developable surfaces: cuspidal-edge (left), swallowtail (center), and cuspidal cross-cap (right).}
  \label{singularity_circ}
\end{figure}

These surfaces exhibit singularities along the base curves (see Fig.\ref{fig:tds} ), \cite{A.Gray}.
The base curve of a tangent-developable surface is also its striction curve.

\begin{itemize}
    \item A tangent developable surface is considered a cuspidal edge surface along $\alpha(t)$ if $\tau(t) \neq 0$.
    \item The tangent developable of a space curve possesses a cuspidal cross-cap singularity at $\alpha\left(t_0\right)$ if $\tau\left(t_0\right) = 0$ and $\tau^{\prime}\left(t_0\right) \neq 0$.
\end{itemize}

These conditions are considered generic for space curves, making the cuspidal cross-cap singularity a common singularity of tangent developable surfaces.

Now we can check the conditions for space curves that correspond to tangent-developable surfaces. We know the striction curve of the tangent-developable surface carries the singularity of the surface. So these regular curves (\(\mathbf{x}\) in \eqref{ca}) correspond to regular curves \(\alpha\).

Vise-versa, we can construct developable surfaces via the curvature axis of singular space curves.

This time we have a space curve whose tangent vanishes identically at a point $t_0$. Therefore we can state the following theorem:
\begin{Theorem}
Let \(\alpha\) be a space curve that is singular at \(t_0\). Then the developable surface generated by the curvature axis of $\alpha$ has a singularity.
\end{Theorem}

\begin{Example}[3D Twisted Curve]
For example, consider the parametric singular space curve \((t^2,t^3,t^4)\). Obviously the curve has a singularity at the origin. If we calculate the curvature axis of this curve from the equation \eqref{ca}, then we have a cuspidal cross-cap swallowtail developable surface; see Fig.\ref{singularity_ccs}. This is a fascinating result.
\end{Example}

\begin{figure}
   \begin{minipage}{0.5\textwidth}
    \centering
    \includegraphics[height=5cm]{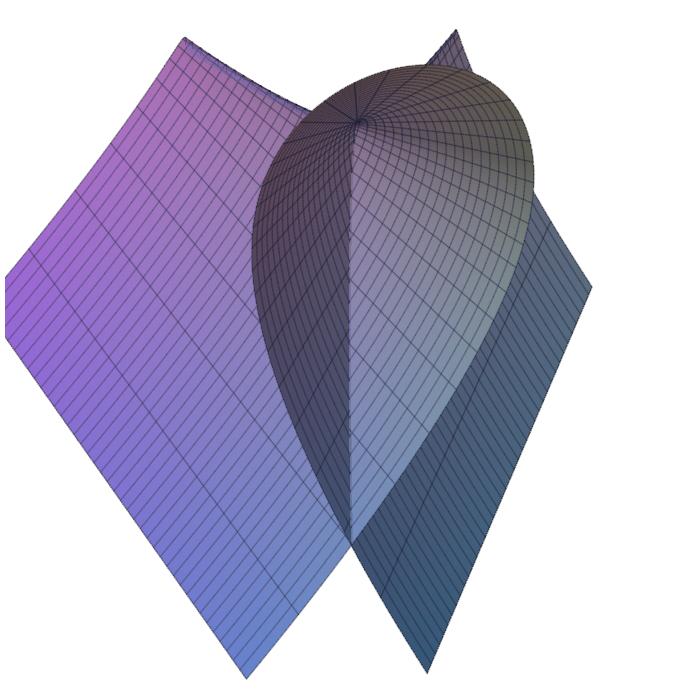} 
  \end{minipage}%
  \begin{minipage}{0.5\textwidth}
    \centering
    \includegraphics[height=5cm]{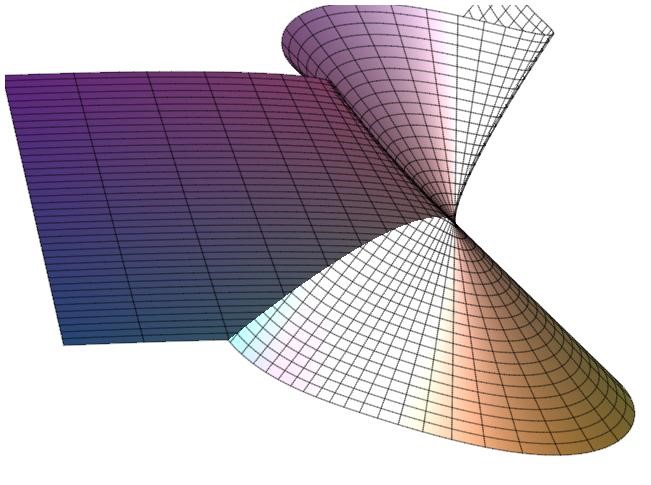} 
  \end{minipage}%
  \caption{Cuspidal cross-cap swallowtail developable surface generated by the curvature axis of the singular space curve \((t^2,t^3,t^4)\).}
  \label{singularity_ccs}
\end{figure}

\begin{figure}
  \centering
  \includegraphics[width=0.9\textwidth]{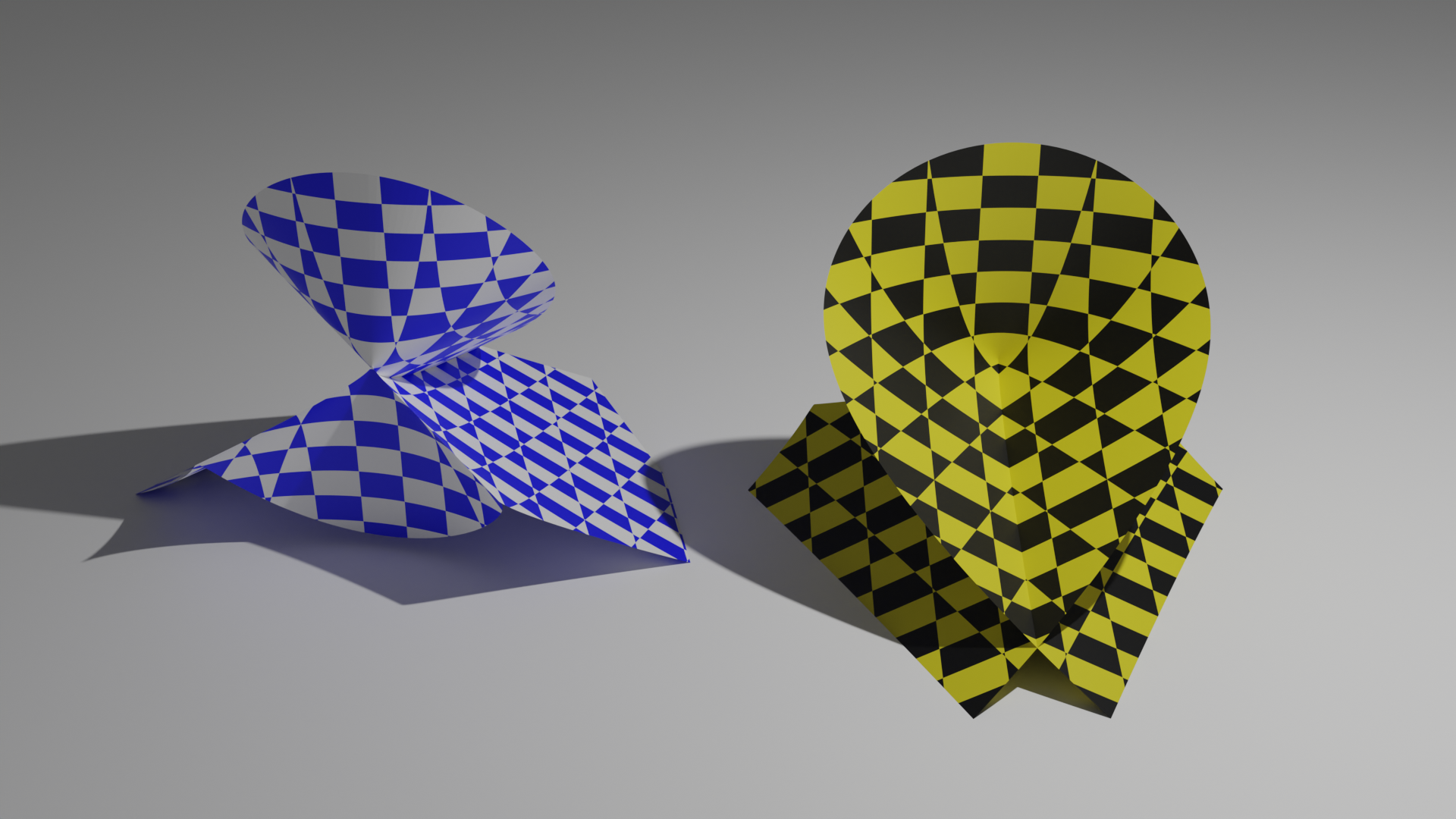}
  \caption{Cuspidal cross-cap swallowtail developable surface visualisation using checker textures, where self-intersection and singularity point are not clearly visible. The example shows that other surface quality methods such as zebra lines or Gaussian curvature visualisation can be used, or an environmental map with a circular pattern can be applied.}
  \label{crosscap}
\end{figure}

We can provide two more examples of singular space curves for the construction of singular developable surfaces.
\begin{Example}[3D Asteroid Curve]
Let $\alpha$ be parametrised as \((\cos(t)^3,\sin(t)^3,\cos(2t))\). Then its singular developable surface is described as follows (see Fig.\ref{astreoid}):
\begin{equation*}
\left(
-\frac{\cos \left(t \right) \left(110 \left(\cos^{2}\left(t \right)\right)-12 \lambda -125\right)}{15}, \frac{\sin \left(t \right) \left(110 \left(\cos^{2}\left(t \right)\right)-12 \lambda +15\right)}{15}, \cos \! \left(2 t \right)-\frac{3 \lambda}{5} 
\right)
\end{equation*}
\end{Example}

\begin{figure}
  \begin{minipage}{0.3\textwidth}
    \centering
    \includegraphics[height=5cm]{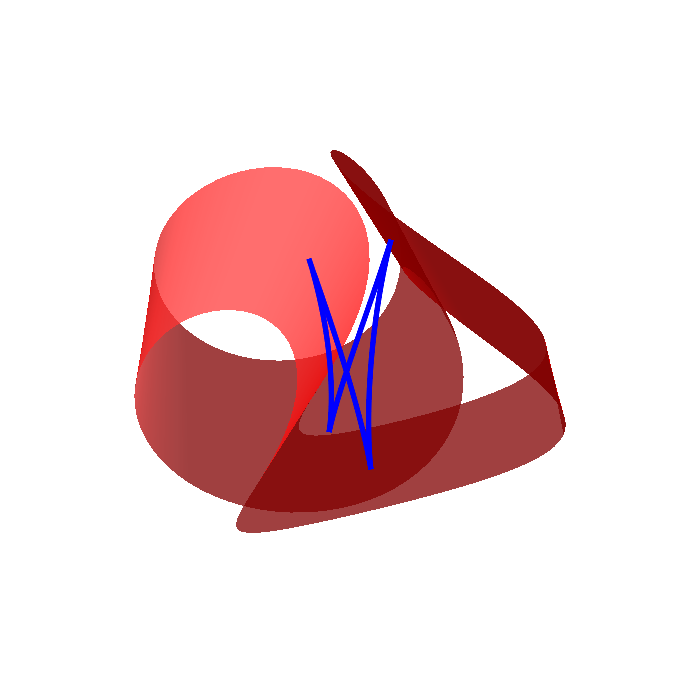} 
  \end{minipage}%
  \begin{minipage}{0.3\textwidth}
    \centering
    \includegraphics[height=5cm]{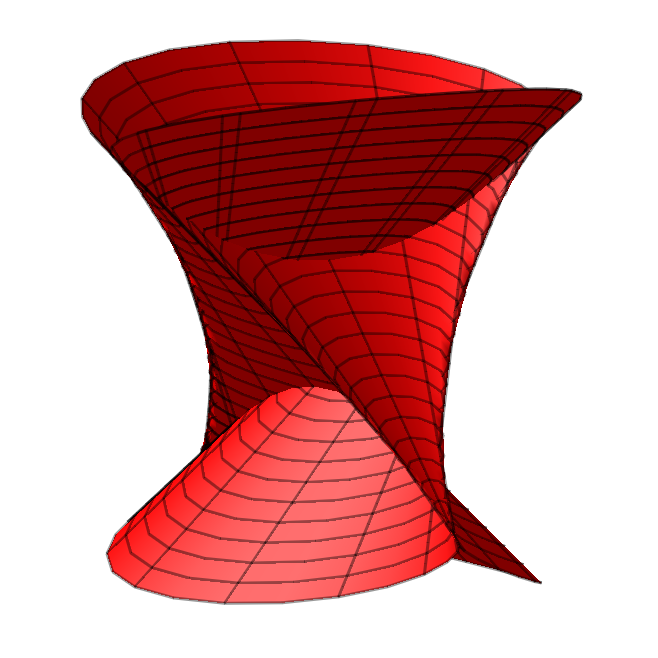} 
  \end{minipage}%
  \begin{minipage}{0.3\textwidth}
    \centering
    \includegraphics[height=5cm]{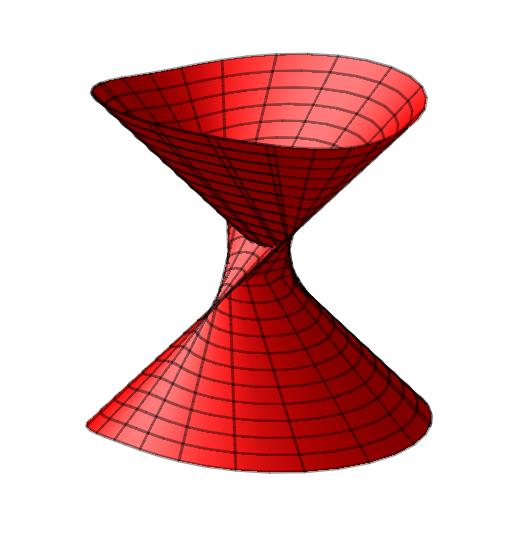} 
  \end{minipage}
  \caption{Singular developable surfaces associated with the 3D asteroid curve (blue). From left to right: Different values of $\lambda$.}
  \label{astreoid}
\end{figure}

\begin{Example}[The Spherical Nephroid Curve]
Let $\alpha$ parametrized as $$
\left(\frac{3}{4} \cos t-\frac{1}{4} \cos 3 t, \frac{3}{4} \sin t-\frac{1}{4} \sin 3 t, \frac{\sqrt{3}}{2} \cos t\right),
$$ then we get a cone-type singular developable surface (see Fig.\ref{fig:nefroid}).
\end{Example}

\begin{figure}
    \centering
    \includegraphics[width=0.5\linewidth]{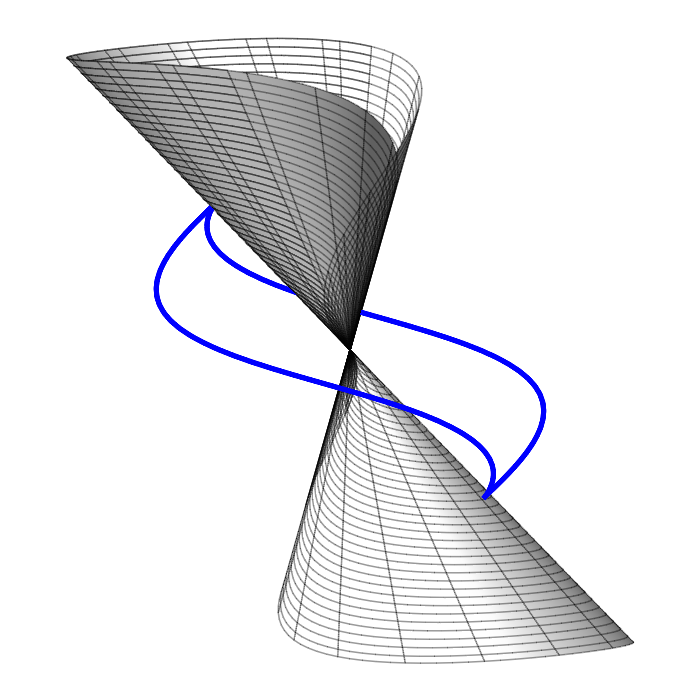}
    \caption{The spherical nefroid and its singular developable surface.}
    \label{fig:nefroid}
\end{figure}

\section{Conclusion and Recommendations for Future Work}\label{conclusion}

Motion is an essential component of our lives. The geometry of lines is one of the most practical areas of geometry, and the motion of the curvature axis is a classic topic of differential geometry. In this study, we examined the motion of the curvature axis of a curve that generates a ruled surface in space. We also demonstrated that for each developable surface, there exists at least one space curve whose curvature axis coincides with any line on the ruled surface. Thus, we established a correspondence between space curves and developable surfaces. Additionally, we shed light on the classifications of developable surfaces corresponding to space curves with singularities. Our research also contributes to the understanding of the singularities of developable surfaces, such as those that can be generated by moving a line along a cuspidal evolute of log-aesthetic curves \cite{yoshida1} approximated by B\'{e}zier curves.

Ruled surfaces are mathematically simple and elegant, and their practical implementation has a long history. However, new types of curves are being discovered that can be used as generators of ruled surfaces in the near future. For example, there is currently an interest in high-quality shapes that include \textit{log-aesthetic curves} \cite{miura1, yoshida1, yoshida2} and \textit{superspirals} \cite{ziatdinov2}, and these curves can be used to generate high-quality and aesthetic ruled surfaces. Such an approach may have applications in automotive and road design, as well as shipbuilding. We would also be interested in applying our work to the study of \textit{umbrella surfaces} \cite{nishikawa1}.

\end{document}